\documentclass{article}

\usepackage{amsmath}
\usepackage{amsthm}

\usepackage[pdftex]{hyperref}

\newtheorem{theorem}{Theorem}[section]

\numberwithin{equation}{section}

\title{The Binomial Coefficient for Negative Arguments}
\author{M.J. Kronenburg}
\date{}

\begin{document}

\maketitle

\begin{abstract}
The definition of the binomial coefficient in terms of gamma functions
also allows non-integer arguments.
For nonnegative integer arguments the gamma functions reduce to factorials,
leading to the well-known Pascal triangle.
Using a symmetry formula for the gamma function, this definition is extended to negative
integer arguments, making the symmetry identity for binomial coefficients
valid for all integer arguments.
The agreement of this definition with some other identities
and with the binomial theorem is investigated.
\end{abstract}


\section{Binomial Coefficients and the Gamma Function}

The definition of the binomial coefficient in terms of gamma functions
for complex $x$, $y$ is \cite{AS70}:
\begin{equation}\label{binomdef}
 \binom{x}{y} = \frac{\Gamma(x+1)}{\Gamma(y+1)\Gamma(x-y+1)}
\end{equation}
For nonnegative integer $n$ and integer $k$ this reduces to \cite{AS70}:
\begin{equation}\label{binomidef}
 \binom{n}{k} =
 \begin{cases}
   \dfrac{n!}{k!(n-k)!} & \text{if $0\leq k\leq n$} \\
   0 & \text{otherwise} \\
 \end{cases}
\end{equation}
The value of the binomial coefficient (\ref{binomdef})
for negative integer $n$ and integer $k$ can be computed
using the following symmetry formula for the gamma function.
\begin{theorem}
For integer $a$, $b$ and complex $s$:
\begin{equation}\label{gammasym}
  \dfrac{\Gamma(s-a+1)}{\Gamma(s-b+1)} = (-1)^{b-a}\dfrac{\Gamma(b-s)}{\Gamma(a-s)}
\end{equation}
\end{theorem}
\begin{proof}
The reflection formula for the gamma function is \cite{AS70,GR07}:
\begin{equation}\label{gammarefl}
 \Gamma(s)\Gamma(1-s) = \frac{\pi}{\sin(\pi s)}
\end{equation}
Combination of this formula with a special case of the addition formula for the sine
function \cite{AS70,GR07}:
\begin{equation}\label{gammasin}
 \sin(\pi(b-s)) = (-1)^{b-a} \sin(\pi(a-s))
\end{equation}
yields the result.
\end{proof}
The recursion formula for the gamma function, which is also used below,
is for complex $s$ \cite{AS70}:
\begin{equation}\label{gammadef}
 \Gamma(s+1) = s\,\Gamma(s)
\end{equation}

\section{Binomial Coefficient for Negative First Argument}

When a gamma function in the definition of the binomial coefficient
(\ref{binomdef}) has nonpositive integer argument, the value
of that gamma function is infinite.
The binomial coefficient (\ref{binomdef}) for nonnegative integer $x$
and integer $y$ can have at most one denominator gamma function
with nonpositive integer argument, in which case the binomial coefficient
becomes zero, and this is expressed by identity (\ref{binomidef}).
For negative integer $x$ and integer $y$ the numerator and at least one of the denominator
gamma functions have nonpositive integer arguments.
In this case (\ref{gammasym}) with $s=0$ can be used to replace two
of these infinite gamma functions with finite ones, which yields computable
expressions for all integers.
This results in the following binomial coefficient identity,
which with identity (\ref{binomidef}) allows computation of the
binomial coefficient for all integer arguments.\\
\begin{theorem}
For negative integer $n$ and integer $k$:
\begin{equation}\label{binominegdef}
 \binom{n}{k} =
 \begin{cases}
   \displaystyle (-1)^k \binom{-n+k-1}{k} & \text{if $k\geq 0$} \\
   \displaystyle (-1)^{n-k} \binom{-k-1}{n-k} & \text{if $k\leq n$} \\
   0 & \text{otherwise} \\
 \end{cases}
\end{equation}
\end{theorem}
\begin{proof}
When $k\geq 0$ the numerator and second denominator gamma functions in (\ref{binomdef})
have nonpositive integer argument. These two gamma functions
are substituted by (\ref{gammasym}) with $s=0$, $a=-n$ and $b=-n+k$,
yielding the first case of (\ref{binominegdef}).
When $k\leq n$ the numerator and first denominator gamma functions in (\ref{binomdef})
have nonpositive integer argument. These two gamma functions
are substituted by (\ref{gammasym}) with $s=0$, $a=-n$ and $b=-k$,
yielding the second case of (\ref{binominegdef}).
When $n<k<0$ all three gamma functions in (\ref{binomdef}) have
nonpositive integer argument. In this case any of the above substitutions
leaves only one denominator gamma function with nonpositive integer argument,
which is infinite, thus yielding the third case of (\ref{binominegdef}).
\end{proof}
Using a different method this was proved earlier in literature \cite{S08}.
The two nonzero cases are mutually transformed by replacing $k$ by $n-k$,
so that the symmetry identity:
\begin{equation}\label{intsym}
 \binom{n}{k} = \binom{n}{n-k}
\end{equation}
is preserved for all integer values of $n$ and $k$.
For derivation of certain binomial coefficient identities
these two expressions are therefore equivalent.\\
When taking $s$ in (\ref{gammasym}) to be a small complex deviation $\delta$ from the integer
arguments $n$ and/or $k$, for each case the following identities result:
\begin{equation}\label{cont1}
 \binom{n+\delta}{k} = (-1)^k \binom{-(n+\delta)+k-1}{k}
\end{equation}
\begin{equation}\label{cont2}
 \binom{n+\delta}{k+\delta} = (-1)^{n-k} \binom{-(k+\delta)-1}{n-k}
\end{equation}
The right side binomial coefficients of these two identities have near nonnegative integer arguments
and are therefore continuous.
The left side binomial coefficients are therefore also continuous at these integer arguments
in these directions in complex \mbox{$x,y$} argument space,
but not in all directions.

\section{Agreement with some Other Identities}

There are some identities which should remain valid under (\ref{binominegdef}).
The symmetry identity \cite{GKP94,K97} and the trinomial revision identity \cite{GKP94,K97}
are valid for all complex $x$, $y$, $z$:
\begin{equation}\label{binomsym}
 \binom{x}{y} = \binom{x}{x-y}
\end{equation}
\begin{equation}\label{binomtri}
 \binom{x}{y}\binom{y}{z} = \binom{x}{z}\binom{x-z}{y-z}
\end{equation}
These two identities follow directly from definition (\ref{binomdef}),
and remain valid under (\ref{binominegdef}) \cite{S08}.
The absorption identity \cite{GKP94,K97} is valid for all
complex $x$, $y$ (except $y=0$, see below):
\begin{equation}\label{binomabs}
 \binom{x}{y} = \frac{x}{y}\binom{x-1}{y-1}
\end{equation}
The addition identity \cite{GKP94,K97} is valid for all
complex $x$, $y$ (except \mbox{$x=y=0$}, see below and \cite{S08}):
\begin{equation}\label{binomadd}
 \binom{x}{y} = \binom{x-1}{y} + \binom{x-1}{y-1}
\end{equation}
These two identities follow from definition (\ref{binomdef})
and gamma function property (\ref{gammadef}),
and remain valid under (\ref{binominegdef}) with the exceptions mentioned.
These exceptions follow from the fact that gamma function
property (\ref{gammadef}) with $s=0$ yields \mbox{$1=0\cdot\infty$},
which when true is non-associative.
This explains that in terms of gamma functions the absorption identity (\ref{binomabs})
is always true, but in terms of binomial coefficients it is always true except for $y=0$.
When the binomial coefficients in the addition identity
(\ref{binomadd}) are substituted with gamma functions,
and the common factor is eliminated using the gamma function property
(\ref{gammadef}), then the following equation results:
\begin{equation}\label{binomaddeval}
 \frac{x}{y(x-y)}\frac{\Gamma(x)}{\Gamma(y)\Gamma(x-y)}
 = \left[ \frac{1}{y} + \frac{1}{x-y} \right] \frac{\Gamma(x)}{\Gamma(y)\Gamma(x-y)}
\end{equation}
which is trivially true when $x\neq 0$ or $y\neq 0$.
For the non-trivial case this expression should be evaluated in the limit to \mbox{$x=y=0$}.
From the gamma function property (\ref{gammadef}) follows that for nonnegative integer $n$ at $x=0$:
\begin{equation}\label{gammaneg}
 \frac{1}{\Gamma(x-n)} = (-1)^n n! \, x + O(x^2)
\end{equation}
which means that the following may be substituted at $x=y=0$:
\begin{equation}
 \frac{\Gamma(x)}{\Gamma(y)\Gamma(x-y)} = \frac{y(x-y)}{x}
\end{equation}
The left side of (\ref{binomaddeval}) reduces to $1$, and the right side
becomes:
\begin{equation}
 \left[ \frac{1}{y} + \frac{1}{x-y} \right] \frac{y(x-y)}{x}
 = \frac{x-y}{x} + \frac{y}{x}
\end{equation}
The right side of this equation evaluates to $1$, except when the limit
to \mbox{$x=y=0$} is taken for both terms separately, because then for both terms
$0/0=1$ and the sum evaluates to $2$.
This is a consequence of the fact that addition and division
are non-associative for infinitely small numbers.
This explains that in terms of gamma functions the addition identity (\ref{binomadd})
is always true, but in terms of binomial coefficients it is always true except for \mbox{$x=y=0$}.\\

\section{Agreement with the Binomial Theorem}

The binomial theorem for nonnegative integer power \cite{AS70,GR07}
defines the binomial coefficients of nonnegative integer arguments in terms of a finite
series, which is the Taylor expansion of \mbox{$x+y$} to the power $n$ in terms of $x$ at \mbox{$x=0$}.\\
For nonnegative integer $n$ and complex $x$, $y$:
\begin{equation}\label{binomsumpos}
 (x+y)^n = \sum_{k=0}^{n} \binom{n}{k} y^{n-k} x^{k}
\end{equation}
Changing all $k$ into $n-k$ in the summation term reverses the direction of summation
and yields this identity with $x$ and $y$ interchanged, which is the Taylor expansion
in terms of $y$ at \mbox{$y=0$}.
Therefore for nonnegative integer $n$ the Taylor expansions in terms of $x$ or $y$ are
identical finite series.
For negative integer power $n$ there is the Taylor expansion in terms of $x$ at \mbox{$x=0$}.\\
For negative integer $n$ and complex $x$, $y$:
\begin{equation}\label{binomsumnega}
 (x+y)^{n} = \sum_{k=0}^{\infty} (-1)^k \binom{-n+k-1}{k} y^{n-k} x^{k}
\end{equation}
In this identity appear the binomial coefficients of the first case of (\ref{binominegdef}).
This series converges only when \mbox{$|x|<|y|$}.
The binomial coefficients of the second case of (\ref{binominegdef}) yield a different identity.\\
For negative integer $n$ and complex $x$, $y$:
\begin{equation}\label{binomsumnegb}
\begin{split}
 (x+y)^{n} & = \sum_{k=n}^{-\infty} (-1)^{n-k} \binom{-k-1}{n-k} y^{n-k} x^{k}\\
           & = \sum_{k=0}^{\infty} (-1)^k \binom{-n+k-1}{k} x^{n-k} y^k
\end{split}
\end{equation}
This identity is (\ref{binomsumnega}) with $x$ and $y$ interchanged,
and therefore the Taylor expansion in terms of $y$ at \mbox{$y=0$}.
This series converges only when \mbox{$|x|>|y|$}.
For negative integer $n$ the Taylor expansions in terms of $x$ or $y$
are two different infinite series.
When taking \mbox{$y=1$} in (\ref{binomsumpos}) to (\ref{binomsumnegb}),
it appears that for nonnegative integer $n$ there is
one finite series for any $x$, but for negative integer $n$ there are two infinite series,
one converging only when \mbox{$|x|<1$} and one only when \mbox{$|x|>1$}.
These two series are represented by the two cases in identity (\ref{binominegdef}).\\

\section{Conclusion}

The binomial coefficient identities (\ref{binomdef}), (\ref{binomidef}) and
(\ref{binominegdef}) define the binomial coefficient as a continuous function
for all complex (including all integer) arguments,
except for negative integer $x$ and non-integer $y$, in which case the binomial coefficient
is infinite.
This definition is in agreement with the binomial theorem.
With this definition the identities (\ref{binomsym}) to (\ref{binomadd}) are always true
in terms of gamma functions,
albeit with the exceptions mentioned in (\ref{binomabs}) and (\ref{binomadd})
in terms of binomial coefficients.
When using (\ref{binomabs}) or (\ref{binomadd}), for example in the derivation of
combinatorial identities or other applications, it is important to be aware of these exceptions,
although they may involve only binomial coefficients with one or two zero arguments
and may therefore not be a problem in practical cases.
Because symmetry (\ref{intsym}) and continuity in certain directions
(\ref{cont1}) and (\ref{cont2}) may be more important than these exceptions,
this definition may be preferred above other definitions.

\pdfbookmark[0]{References}{}


\begin{thebibliography}{99}
\bibitem{AS70}
  M. Abramowitz, I.A. Stegun,
  \textit{Handbook of Mathematical Functions},
  Dover, 1970.
\bibitem{GR07}
  I.S. Gradshteyn, I.M. Ryzhik,
  \textit{Table of Integrals, Series and Products}, 7th ed.,
  Academic Press, 2007.
\bibitem{GKP94}
  R.L. Graham, D.E. Knuth, O. Patashnik,
  \textit{Concrete Mathematics, A Foundation for Computer Science}, 2nd ed.,
  Addison-Wesley, 1994.
\bibitem{K97}
  D.E. Knuth,
  \textit{The Art of Computer Programming, Volume 1: Fundamental Algorithms}, 3rd ed.,
  Addison-Wesley, 1997.
\bibitem{S08}
  R. Sprugnoli,
  Negation of binomial coefficients,
  \textit{Discrete Math.} 308~(2008)~5070-5077.
\end{thebibliography}
\end{document}